\documentclass{amsart} 
\usepackage{graphicx} 
 \usepackage{color}  
 \usepackage{appendix} 
\usepackage{amsmath, amssymb, mathrsfs, amsfonts,  amsthm} 
\usepackage[all]{xy}

\newcommand{\delbar}{\bar{\partial}}

\DeclareMathOperator{\area}{area}
\DeclareMathOperator{\ind}{index}

\newtheorem {theorem} {Theorem} [section]

\newtheorem{lemma}[theorem] {Lemma}

\newtheorem {definition} [theorem] {Definition} 
\newtheorem {proposition}  [theorem]{Proposition} 

\newtheorem {corollary}[theorem]  {Corollary} 
\newtheorem {example} [theorem]  {Example}

\newtheorem {remark} [theorem] {Remark}
\numberwithin {equation} {section}

\begin{document} 
\author {Yasha Savelyev} 
\title {Morse theory for the Hofer length functional} 

\address{ICMAT Madrid}
 \email {yasha.savelyev@gmail.com}  
\begin{abstract}  Following
   \cite{citeSavelyevVirtualMorsetheoryon$Omega$Ham$(Momega)$.}, we develop here a connection
between Morse theory for the (positive) Hofer length functional $L: \Omega \text
{Ham}(M, \omega) \to \mathbb{R}$, with Gromov-Witten/Floer theory, for  monotone
symplectic manifolds $ (M, \omega) $. This gives some immediate  restrictions on the topology of
the group of Hamiltonian symplectomorphisms (possibly relative to the Hofer
length functional), and a criterion for non-existence of certain higher index
geodesics for the Hofer length functional.  The argument is based 
on a certain automatic transversality
phenomenon which uses Hofer geometry to conclude transversality and may be
useful in other contexts.  Strangely the monotone assumption seems essential for
this argument, as abstract perturbations necessary for the virtual moduli
cycle, decouple us from underlying Hofer geometry, causing automatic transversality
to break.
%  In \cite{virtmorse} we studied a connection 
% This lead to some interesting geometry, but was  limited in immediate
% applications to generalized full flag manifolds $M = G/T $, with $G $ a semi-simple Lie group
% and $T $ its maximal torus, and circle actions for critical points. Our goal
% here is to give a general theorem,  for every monotone symplectic manifold $ (M, \omega) $ and
%  much more general Ustilovsky geodesics as critical points. 
 \end {abstract}
\maketitle 
\section{Introduction} 
Topology of the group of Hamiltonian diffeomorphisms $ \text {Ham}(M, \omega) $
of a symplectic manifold $ (M, \omega) $, is a very rich object of study
connected to all techniques of modern symplectic topology. The first major
investigations already appear in Gromov's
\cite{citeGromovPseudoholomorphiccurvesinsymplecticmanifolds.} for some four
dimensional symplectic manifolds using Gromov-Witten theory. Much of the subsequent study of
the subject was based on this. In the higher dimensional setting rather little
is known outside special cases with high symmetry, e.g.
\cite{citeReznikovCharacteristicclassesinsymplectictopology}.
The problem in general is that it is very hard to even construct good candidates for ``cycles''
in $ \text {Ham}(M, \omega) $, which may be non-trivial. For $\pi_1 $
there is one very natural candidate: Hamiltonian circle actions. It turns out
\cite{citeMcDuffTolmanTopologicalpropertiesofHamiltoniancircleactions} that if the circle action is in appropriate sense semi-free, it always
represents a non-trivial class in $ \pi_1 $. There are some
generalizations of this to certain geodesics of the Hofer length
functional  for example
\cite{citeMcDuffSlimowitzHofer--ZehndercapacityandlengthminimizingHamiltonianpaths}. In these cases one
crucial necessary condition on such a geodesic is that it must be index 0.  
From this point of view it is in a sense clear how to try generalize the
above: consider more general higher index geodesics for the (positive) Hofer
length functional, and their unstable manifolds, i.e. try to do some kind of  Morse
theory. One immediate problem is that  in general we can only make sense of 
``unstable manifolds'' locally, (see however
\cite{citeSavelyevVirtualMorsetheoryon$Omega$Ham$(Momega)$.} for examples of
when global unstable manifolds do exist) and so must work with relative classes. Also
to guarantee local smoothness of the Hofer length functional,  we must restrict the class 
of geodesics to what we call ``Ustilovsky geodesics'', which first appear in
\cite{citeUstilovskyConjugatepointsongeodesicsofHofersmetric}, 
and whose theory is further developed in
\cite{citeLalondeMcDuffHofers$Linfty$--geometry:energyandstabilityofHamiltonianflowsIII} in more
generality. Nevertheless the local Morse theory can be set up. This gives us
candidates for ``cycles'', when are they non-trivial? We show that this always 
happens under certain Floer theoretic assumptions on the Ustilovsky
geodesic,  which leads us to
the notion of ``robust Ustilovsky geodesic'', (technically we still have to
perturb the geodesic). This is a strange phenomenon. The robust
condition is some global Floer theoretic condition, but is local from the 
point of view of $ \text {Ham}(M, \omega) $, 
yet it is enough to deduce the global fact that 
the above cycles are non-trivial.  
 
The main argument is based on a certain  unusual automatic transversality
phenomenon, which actually uses Hofer geometry to conclude transversality. 
This already appeared in
\cite{citeSavelyevVirtualMorsetheoryon$Omega$Ham$(Momega)$.} in less generality, but was somewhat
obscured and had some inaccuracies.
\subsection {The group of Hamiltonian symplectomorphisms and Hofer metric}
\label{sec.hofer} Given a smooth function $H: M ^{2n} \times (S ^{1} =
\mathbb{R}/ \mathbb{Z}) \to \mathbb{R}$,  there is  an associated time dependent
Hamiltonian vector field $X _{t}$, $0 \leq t \leq 1$,  defined by 
\begin{equation}  \label{equation.ham.flow} \omega (X _{t}, \cdot)=-dH _{t} (
\cdot).
\end{equation}
The vector field $X _{t}$ generates a path $\gamma: [0,1] \to \text
{Diff}(M)$, starting at $id$. Given such a path $\gamma$, its end point
$\gamma(1)$ is called a Hamiltonian symplectomorphism. The space of Hamiltonian symplectomorphisms
forms a group, denoted by $\text {Ham}(M, \omega)$.
In particular the path $\gamma$ above lies in $ \text {Ham}(M, \omega)$. It
is well-known 
that any smooth path $\gamma$ in $\text {Ham}(M, \omega)$ with
$\gamma (0)=id$ arises in this way (is generated by $H: M \times [0,1] \to
\mathbb{R}$ as above). Given  a general smooth path $\gamma$, the \emph{Hofer
length}, $L (\gamma)$ is defined by 
\begin{equation*} L (\gamma):= \int_0 ^{1} \max _{M} H _{t} ^{\gamma} 
-\min _{M} H ^{\gamma} _{t} dt,
\end{equation*}
where $H ^{\gamma}$ is a generating function for the path
$t \mapsto \gamma({0}) ^{-1} \gamma (t),$ $0\leq t \leq 1$. 
% , satisfying $\int _{M} H ^{\phi}
% _{t}\omega ^{n}=0$ for all $t$. 
The Hofer distance $\rho (\phi, \psi)$ is  defined by taking the
infinum of the Hofer length of paths from $\phi $ to $\psi$. We only mention
it, to emphasize that it is a deep and interesting theorem that the resulting
metric is non-degenerate, (cf.
\cite{citeHoferOnthetopologicalpropertiesofsymplecticmaps,
citeLalondeMcDuffThegeometryofsymplecticenergy}). This gives $ \text {Ham}(M, \omega)$ the structure of a Finsler
manifold. 
A related functional, $L ^{+} $, that more readily connects to Gromov-Witten
theory is given by
 \begin{equation*} L ^{+} (\gamma)   := \int_0 ^{1} \max _{M} H _{t} ^{\gamma}
 dt,
\end{equation*}
for $H _{t} ^{\gamma} $ the generating function normalized by $\int _{M} H _{t}
^{\gamma} \omega ^{n} = 0$, for every $t $.

We now consider $L ^{+}$ as a functional on the space of paths in $ \text
{Ham}(M, \omega)$ starting at $id$ and ending at some fixed end point $\phi$, denote
this by $ \mathcal {P} _{\phi}$.
% Let  $o: S ^{1} \to \text {Ham}(M,
% \omega)$, be a circle subgroup generated by a Morse, time independent
% Hamiltonian $H$. And  set $\gamma = o| _{ [\epsilon, 1- \epsilon]}$,  where  $
% \epsilon$ is sufficiently small, (we will make this precise in the course of the proof).
It is shown by Ustilovsky
\cite{citeUstilovskyConjugatepointsongeodesicsofHofersmetric} that $ \gamma$ is a smooth critical
point of $$L ^{+}:  \mathcal {P} _{\phi} \to
\mathbb{R},$$ if there is a unique point $x _{\max} \in M$
maximizing the generating function $H ^{\gamma} _{t}$
at each moment $t$, and such that $H ^{\gamma} _{t}$ is Morse at $x _{\max}$, at
each moment $t $.
 \begin{definition} We call such a $\gamma$ an \emph {
\textbf{Ustilovsky geodesic}}.
\end{definition}
% Ustilovsky's theory has a very simple generalization to the case  the generating function $H
% ^{\gamma} _{t}$, is Morse at a unique extremizer $x _{\max} $, except for $t $
% where $H ^{\gamma} _{t} = 0$. Such a $\gamma $ is a smooth point for the Hofer
% length functional. 
\begin{definition} Given  a chain complex $(A _{\bullet},d)$ with some                                 
 distinguished basis, and  the inner product 
determined by this basis, (with respect to which it is orthonormal), we 
say that a chain $c$ is \emph {\textbf{semi homologically essential}} if 
$c $ is orthonormal to $ d (e) $ for any $e$. 
%  if there is nos
% quasi-isomorphic sub-complex of  $A _{\bullet} $, which does not contain $c $. 
\end{definition}
This is of course automatic if $A _{\bullet}$ is perfect, which is often the
case in Floer theoretic applications we consider.  
 Although we don't need this, it is relatively easy to see that
if $c $ is semi-homologically essential in a chain complex of vector spaces 
$(A _{\bullet}, d) $, and is closed ($dc=0 $)
then it is homologically essential, meaning there is no quasi-isomorphic 
sub-complex of $A
_{\bullet} $ which is orthogonal to $c $. If $c$ is homologically essential then
again a bit of elementary algebra implies that it is semi-homologically
essential. (We need field coefficients for this.) 
% \begin{lemma} \label{lemma.essential} If $A _{\bullet}$ is a chain complex of
% vector spaces, a generator $c \in A _{\bullet}$ is homologically essential if and only if
%  $c $ is not a component of a differential $d(e) $ for any $e $ and $d (c) =0$.
% \end{lemma}
\begin{definition}  We will say that an Ustilovsky geodesic $\gamma \in \mathcal
{P} _{\phi} $ is \emph{\textbf{robust}}, if $\widetilde{\phi}$ is Floer
non-degenerate and the constant, period one orbit $o _{\max}$ at $x
 _{\max} $ for the flow $\gamma $ is  semi-homologically essential in $CF
 (\gamma) $. Here $ \widetilde{\phi} \in \widetilde{ \text {Ham}} (M, \omega)$
 is the lift of $ \phi $ to the universal cover determined by $\gamma $.
\end{definition} 
\begin{theorem} \label{thm.indexUst}
   \cite{citeSavelyevProofoftheindexconjectureinHofergeometry}
Let $\gamma \in \mathcal {P} _{\phi}$
be an Ustilovsky geodesic, then  the Morse index of $ \gamma$ with respect to
 $L _{+}$ is
\begin{equation} \label{eq.difference} 
|CZ (o _{\max}) - CZ ( [M])|, 
\end{equation}
where $CZ $ is the Conley-Zehnder index and $CZ ( [M]) $ denotes the
Conley-Zehnder degree of the fundamental class, under the PSS isomorphism. 
\end{theorem}

If the Conley-Zehnder index is normalized as in
\cite{citeRobbinSalamonTheMaslovindexforpaths.}, then
 by the proof of \ref{thm.indexUst}, $CZ (o_{\max})
- CZ ( [M]) \leq 0$. (With $CZ ( [M]) = -n$.) This is the normalization that will
be assumed for the grading of the Floer chain complex.  (Although it will be
implicit.)
\subsection {Statement of main theorems}

\begin{definition} \label{def.unstable} Let $\gamma \in \mathcal {P} _{\phi}$ be
   an index $k$ Ustilovsky geodesic and let $B ^{k} $ denote the standard
   $k$-ball in $ \mathbb{R} ^{n} $, centered at the origin 0. Let $ \mathcal {P}
   _{\phi,  E ^{\gamma}} $ denote the $ E ^{\gamma}$ sub-level set of $ \mathcal
   {P} _{\phi}$, with respect to $L ^{+} $, with $0 < E ^{\gamma} < L ^{+}
   (\gamma)$.  \emph { \textbf{A local unstable manifold}} at $\gamma $ is a
   pair $ (f _{\gamma}, E ^{\gamma}) $, with  $f _{\gamma}: B ^{k} \to \mathcal
   {P} _{\phi}$, s.t. $f _{\gamma} (0)  = \gamma$, $f _{\gamma} ^{*} L ^{+} $ is
   a function Morse at the unique maximum $0 \in B ^{k}$, and s.t. $f _{\gamma}
   (\partial B) \subset \mathcal {P} _{\phi, E ^{\gamma}}$, \end{definition}

Note that local unstable manifolds for a given Ustilovsky geodesic always exist
as can be immediately deduced from \cite[2.2A]{citeUstilovskyConjugatepointsongeodesicsofHofersmetric}, for geodesics coming from
circle actions  elegant explicit local unstable manifolds are constructed in
\cite{citeKarshonSlimowitzShorteningtheHoferlengthofHamiltoniancircleactions}.
\begin{theorem} \label{thm.virtual} Suppose that $ (M, \omega) $ 
   is a monotone symplectic manifold and $\gamma \in \mathcal {P}
_{\phi}$ is a Morse index $k$, robust Ustilovsky geodesic. 
Then there is a $\gamma' \in \mathcal {P} _{\phi'}$, which is a robust Ustilovsky geodesic
arbitrarily $C ^{\infty} $-close to $ \gamma $, a number $E ^{\gamma'} < L ^{+}
(\gamma') $ arbitrarily close to $L ^{+} (\gamma') $ and $(f _{\gamma'}, E
^{\gamma'})$ a local unstable manifold at $\gamma' $
, s.t. \begin{equation*} 0 \neq [f _{\gamma'}] \in  \pi _{k} ( \mathcal {P} _{\phi'},
\mathcal {P} _{\phi',  E ^{\gamma'}}) \otimes \mathbb{Q}.
\end{equation*}
\end{theorem}
 Below we list some examples of robust Ustilovsky geodesics.
 \begin{example} \label{ex.robust} 
An Ustilovsky geodesic $\gamma \in \mathcal {P} _{\phi} $ is robust, for example
if $ CF (\gamma) $ is perfect.  An explicit example for $M=S ^{2}$ with
arbitrary index   could be obtained as follows. Suppose $H$ generates a $k+1/2 $
fold rotation of $S ^{2} $.  In this case $H$ is Floer non-degenerate and $CF
(\gamma) $ is perfect simply by virtue of not having odd degree generators. (It
is generated by $o _{\max}$, $o _{\min} $ as a module over the Novikov ring.)
 The index of $\gamma $ is $2k$.
\end{example}
\begin{example} \label{ex.main}
% Given existence of $S ^{1}$ virtual localization in Floer theory,
We may generalize the above example  as follows.  Suppose we have a symplectic manifold
   $M$ and $H$ a Morse Hamiltonian generating a semi-free circle action in time 1. Here
semi-free means that the isotropy group of every point in $M$ is either trivial
or the whole group.  Then the time 1 flow map for $\tau \cdot H$, with $\tau \not\in
\mathbb {N}$ is Floer
non-degenerate, and has no non-constant period 1 orbits.
The CZ index of all the constant period 1 orbits, i.e. critical points of $H$
 in this case must be even,
which can be readily checked, as the linearized flow at the critical points is a
path in $U (n)$. (Linearizing the circle action itself we get an $S ^{1}$
subgroup of $Symp (\mathbb{R} ^{2n})$ which must be unitary as $U (n)$ is the
maximal compact subgroup of $Symp(\mathbb{R} ^{2n})$). 
In particular the Floer complex is perfect, cf.
\cite{citeKermanLalondeLengthminimizingHamiltonianpathsforsymplecticallyasphericalmanifolds.(Diffeotopieshamiltoniennesminimisantesdanslesvarietessymplectiquementaspheriques.)}, 
\cite{citeMcDuffSlimowitzHofer--ZehndercapacityandlengthminimizingHamiltonianpaths}. 
\end{example}
We may consequently get a more explicit version of \ref{thm.virtual} as follows.
\begin{theorem} \label{thm.special.case} Suppose $\gamma$ is an index $ k $
Ustilovsky geodesic, which in addition is a path determined by a Hamiltonian circle action, (is a
restriction thereof). If $ (f_\gamma, E _{\gamma}) $ is a local unstable manifold for $\gamma $ then
$$0 \neq [f _{\gamma}] \in  \pi _{k} ( \mathcal {P} _{\phi},
\mathcal {P} _{\phi,  E ^{\gamma}}).$$
\end{theorem}
This follows by \ref{thm.robustustilovsky},  \ref{example.satisfied}, and \ref{ex.main}.
It  is worth pointing out for comparison that the  length/energy functional on
the path space (with fixed end points) of a smooth Riemannian manifold $X,g$,
may have lots of critical points (i.e. geodesics) which do not satisfy the
analogue of the property above.  It is tricky to describe higher index examples
without getting completely side-tracked. We can for example produce a Riemannian
manifold, and a submanifold $U$ of the path space with $U$ shaped like a heart
shaped sphere for the energy functional, with non-degenerate critical points
(geodesics) of index 2,1,2,0.  The index 1 geodesic in this case clearly does
not have the analogue of the property of Theorem \ref {thm.virtual} above. In
fact it is possible to show that if all the geodesics of $X,g$ satisfy the
property above and are all non-degenerate then the energy functional is perfect,
i.e. the $i$'th Betti number of the path space is the number of index $i$
geodesics.  In the index 0 case here is a very elementary concrete example.
Deform the $z=0$ isometric embedding of $ \mathbb{R} ^{2} $ into $ \mathbb{R}
^{3}$ so that the image acquires a single mountain (and is flat elsewhere).
Pull-back the metric.  Take points $p,q \in \mathbb{R} ^{2}$ to be on the
opposite side of the mountain.  There are a pair of geodesics going around the
base of the mountain.  We may shape the mountain,  so that none of their
sufficiently small perturbations are length minimizing.
% A criterion for the unstable manifold to represent a non-trivial relative
% homotopy class is given in Theorem \ref{thm.robustustilovsky}. In the case $k=0
% $, it can be sharpened as follows. 

Theorem \ref{thm.virtual} of course also give restrictions on absolute homotopy
groups. Considering the classical long exact sequence for
relative homotopy groups (tensored with the flat $ \mathbb{Z} $ module $
\mathbb{Q}$) we get the following: 
\begin{corollary} \label{corollary.absolute} Under assumptions of the theorem
above, either $\pi_{k-1} (\mathcal {P} _{\phi',  E ^{\gamma'}}, \mathbb{Q})
\neq 0$ or $\pi_k ( \mathcal {P} _{\phi}, \mathbb{Q}) \simeq \pi _{k+1} ( \text
{Ham}(M, \omega), \mathbb{Q}) \neq 0$.
\end{corollary}
More explicitly, if the boundary of the local unstable manifold $(f _{\gamma'},
E _{\gamma'})$ can be contracted inside $\mathcal {P} _{\phi',  E ^{\gamma'}}$ we
get a sphere representing a non-trivial class in $\pi_k (\mathcal {P} _{\phi},
\mathbb{Q}) \simeq \pi _{k-1} (\text {Ham}(M, \omega), \mathbb{Q})$.

Here is one concrete  corollary: \begin{corollary} \label{corollary} Given a
   local unstable manifold $ (f_\gamma, E _{\gamma}) $ at a robust index $k>2 $
   Ustilovsky geodesic $\gamma $ in the path space of $ \text {Ham}(S ^{2}) $, $
   \mathcal {P} _{\phi} $, and $\gamma' $ as in \ref{thm.virtual}, the class of
   the boundary map $[\partial f _{\gamma'}]: B ^{k-1} \to \mathcal {P} _{\phi',
      E _{\gamma'}} $ is non-zero in $\pi _{k-1} ( \mathcal {P} _{\phi', E
      _{\gamma'}}, \mathbb{Q})$. The same holds for $M = \mathbb{CP} ^{2} $ if
   $k>4$.  \end{corollary} This follows upon noting that for $k>2$, $\pi _{k} (
\mathcal {P} _{\phi'}, \mathbb{Q}) = \pi _{k+1} \text {Ham}(S ^{2}, \mathbb{Q})
=0$, and for in case of $M = \mathbb{CP} ^{2} $ for $k>4 $, $\pi _{k} ( \mathcal
{P} _{\phi'}, \mathbb{Q}) = \pi _{k+1} (\text {Ham}( \mathbb{CP} ^{2}),
\mathbb{Q}) = \pi _{k+1} (PSU (3), \mathbb{Q}) = 0.$ (Note that $PSU (3)$ does
have non-vanishing rational homotopy group in degree 5.)

Note that there are robust Ustilovsky geodesics in $ \text {Ham}(S ^{2}) $ 
of arbitrary even index, and $\phi, \phi' $ can be taken to be
arbitrarily close to $id$ see Example \ref{ex.robust}. In particular although
$\Omega \text {Ham}(S ^{2}) \simeq \Omega SO (3)$ has vanishing rational homotopy groups
 in degree greater than 2, there are $E $ sub-level sets for the
(positive) Hofer length functional, with $E$ arbitrarily large, which do not. 

 Here is another  geometric/dynamical
application of \ref{thm.virtual}, formally reminiscent of the non-existence
result
\cite{citeLalondeMcDuffHofers$Linfty$--geometry:energyandstabilityofHamiltonianflowsIII} for length minimizing Hofer geodesics in $ \text {Ham}(S ^{2})
$.
\begin{corollary} \label{thm.deltaminimizing} 
% For $\phi \in \text {Ham}(S ^{2}) $,
% or $\phi \in \text {Ham}( \mathbb{CP} ^{2}) $ with $d _{H} (id, \phi) < 1/2 - \epsilon $, for $0 <\epsilon < 1/2 $.
% . The conclusion of the theorem holds
Given a monotone
$(M, \omega)$, and $\phi \in \text {Ham}(M, \omega) $ s.t. the space of
paths $\delta$ close to minimizing the positive Hofer length (a.k.a $\delta
$-minimizing paths) from $id $ to $\phi$ has vanishing rational homotopy
groups, for some $\delta>0 $,  there is no index $k >0$ robust Ustilovsky geodesic from identity to $\phi $, $ \delta $-close to
   minimizing the Hofer length.
\end{corollary}
We expect that the condition on $\phi $ holds for any $\phi $
sufficiently close to identity. For $ \text {Ham}(S^{2}) $ this would follow for example if it was known
that $ \text {Ham}(S ^{2}) $ had some contractible epsilon ball with respect to
the (positive) Hofer metric. (Having vanishing rational homotopy groups
would suffice for the above.) Although at the moment it is not clear how to
verify this, in a joint work in progress with Misha Khanevsky we intend to show that
 there is $\epsilon _{k} >0 $ and a Hofer $\epsilon _{k}$-ball $B
_{\epsilon _{k}} $ in the space of Lagrangians in $S ^{2} $ Hamiltonian isotopic
to the equator, with the number of intersections with the equator at most $k$,
such that $B _{\epsilon _{k}} $ is contractible. Some initial results in the
spirit of this discussion are obtained in
\cite{citeLalondeSavelyevOntheinjectivityradiusinHofergeometry}, in
particular we show there that for $\phi \in \text {Ham}(S ^{2}) $ sufficiently
close to identity, there is a $\delta $ so that
inclusion map from the space of $\delta $-minimizing paths into the space of all
paths vanishes on rational homotopy groups. 

 As another corollary we have:
\begin{theorem} \label{thm.index0} Let $\gamma \in \mathcal {P} _{\phi}$ be
a robust, index 0, Ustilovsky geodesic, then  $\gamma$ globally minimizes $L ^{+}
$ in its homotopy class, relative to end points.
\end{theorem}
 This result is new, although there are related existing results, for example
 \cite{citeMcDuffSlimowitzHofer--ZehndercapacityandlengthminimizingHamiltonianpaths},
 \cite{citeKermanLalondeLengthminimizingHamiltonianpathsforsymplecticallyasphericalmanifolds.(Diffeotopieshamiltoniennesminimisantesdanslesvarietessymplectiquementaspheriques.)}. These
are essentially based on variations of sufficient conditions on $o _{{\max}} $
for being homologically essential, (in the second case this is generalized to
take into account action intervals). 
%Although our condition on $o _{\max} $ for
%being semi-homologically essential is in principle more loose, see Lemma
%\ref{lemma.semiessential}
% Asmall note on the statement. For $X= diffVOL (M) $ with its $L ^{2} $ metric,
% or indeed for any perhaps infinite dimensional Riemannian manifold, having a
\begin{theorem} \label{thm.minimizing2} Under the conditions of theorem
   \ref{thm.special.case} above, the local unstable manifold $$f _{\gamma}:
   (D^k, \partial D^k) \to (\mathcal {P} _{\phi}, \mathcal {P} _{\phi, E
      ^{\gamma}})$$ realizes the minimum in the definition of the semi-norm:
\begin{equation} | \cdot|: \pi _{k} ( \mathcal {P} _{\phi},
\mathcal {P} _{\phi, {E ^{\gamma}}}) \to \mathbb{R} ^{+},
\end{equation} 
given by $$| [f]| = \inf _{f' \in [f]} \max _{b \in D^k} L ^{+} (f' (b)). $$
\end{theorem}
This is a special case of Theorem \ref{thm.minimizing}.  \subsection {Outline of
   the argument} Our symplectic manifold $ (M, \omega) $ is assumed everywhere
to be monotone. We first construct, using a kind of parametric Floer
continuation map, a group-homomorphism (for $k>0$, otherwise just a set map):
\begin{equation*} \Psi _{\gamma}:   \pi _{k} ( \mathcal {P}_{\phi}, \mathcal {P}
   _{\phi,{E ^{\gamma}}}) \otimes \mathbb{Q} \to \mathbb{Q}, \end{equation*} for
$E ^{\gamma} $ sufficiently close to $L ^{+} (\gamma) $. Here and from now on $
\mathcal {P} _{\phi} $ denotes the component of the path space in the homotopy
class $ [\gamma ]$. As indicated the homomorphism depends on a particular
Ustilovsky geodesic $ \gamma$. The setup for Gromov-Witten theory needed in the
definition of $\Psi _{\gamma} $ is somewhat unusual, and we need to take time to
describe the class of almost complex structures needed for this. Let $ \mathcal
{C} $ denote the space of Hamiltonian connections on $ M \times \mathbb{C} $
satisfying the first pair of conditions in the Definition
\ref{definition.admissible} below.

For $ \mathcal {A} \in \mathcal {C} $ or $
\widetilde{\Omega} $ a coupling 2-form on $M \times \mathbb{C} $,
inducing a connection in $ \mathcal {C} $ define
\begin{align*}
 \area ( \mathcal {A}) &= \inf \{\int _{ \mathbb{C}} \alpha \, |
 \widetilde{\Omega} _{ \mathcal {A}} + \pi ^{*} (\alpha) \text { is
 nearly symplectic} \}, \\
% \in \mathcal {A} _{ \widetilde{\Omega}}\}.
 \area ( \widetilde{\Omega}) &= \inf \{\int _{ \mathbb{C}} \alpha \, |
 \widetilde{\Omega} + \pi ^{*} (\alpha) \text { is
 nearly symplectic} \}, 
\end{align*}  
where $ \widetilde{\Omega} _{ \mathcal {A}} $ is the coupling form inducing $
\mathcal {A} $, (\cite[Theorem 
6.21]{citeMcDuffSalamonIntroductiontosymplectictopology}) $\alpha $ a 2-form on $ \mathbb{C} $
and nearly symplectic means that \begin{equation} \label{eq.almostcomp}
(\widetilde{\Omega} _{ \mathcal {A}} + \pi ^{*} (\alpha)) ( \widetilde{v},
\widetilde{jv}) \geq 0,
\end{equation}
for $ \widetilde{v}, \widetilde{jv} $ horizontal lifts with respect to  $ \mathcal {A}
$, of $v, jv \in T _{z} \mathbb{C}$, for all $ z$. It is not hard to see that
the infinum is attained on a uniquely defined 2-form $\alpha _{ \mathcal {A}} $:
\begin{equation} \label{eq.alphaA} \alpha _{ \mathcal {A}} (v, w) = \max
_{M}  R _{ \mathcal {A}} (v,w),
\end{equation}
where $R _{ \mathcal {A}} $ is the Lie algebra valued curvature 2-form of $
\mathcal {A} $, and we are using the isomorphism $lie \text {Ham}(M,
\omega) \simeq C ^{\infty} _{norm} (M)$, see Section \ref{sec.prelim.coup}. By
assumptions this form has compact support. 
\begin{definition} \label{definition.admissible}  For a $0 < \delta < (L ^{+}
(\gamma) - E ^{\gamma})/2$ and a given $f: (B^k, \partial B^k) \to (\mathcal {P}
_{\phi}, \mathcal {P} _{\phi,{E ^{\gamma}}}) $, a family of Hamiltonian
connections $ \mathcal {A} _{b} $ on $ M \times \mathbb{C} $ is said to be 
$\delta$-\emph{\textbf{{admissible with respect to $f$}}}, if:
\begin{itemize} 
 \item Using the modified polar coordinates $(r,\theta)$, $0 \leq r < \infty$,
    $0\leq \theta \leq 1$, for $r \geq 2$ each $ \mathcal {A} _{b} $ is
flat and invariant under the dilation action of $ \mathbb{R} $ on $ \mathbb{C}
$.
\item The holonomy path $p ( \mathcal {A} _{b}) $  of $ \mathcal {A}
_{b} $ over the circles $ \{r\} \times S ^{1}$ is $f (b) $,  for $r \geq 2$.
\item $|\area ( \mathcal {A}_b) - L ^{+} (f(b))| \leq \delta$.
\end {itemize}
\end {definition}
We will say that $ \mathcal {A} $ is $\delta $-\emph{admissible} with respect to
$p \in \mathcal {P} _{\phi} $, if it satisfies the conditions above with
respect to $p $.
Fix a family $\{j _{r, \theta}\}$, on the
vertical tangent bundle $T ^{vert} (M \times \mathbb{C}) $,  invariant under the
dilation action of $ \mathbb{R} $ for $r>2 $, so that each $j _{r, \theta} $ is compatible
with $\omega $: $\omega (\cdot, j \cdot) >0 $, for $\cdot \neq 0 $.  Then a Hamiltonian
connection $ \mathcal {A} $ on $ M \times \mathbb{C} $ induces an almost 
complex structure $J ^{ \mathcal {A}} $ on $M \times
\mathbb{C} $ having the properties:
\begin{itemize} 
  \item Each $J ^{ \mathcal {A}}
$ coincides on the vertical tangent distribution of $M \times \mathbb{C} $ with
$ \{j _{r, \theta}\}$.
\item The projection map $\pi: M \times \mathbb{C} \to \mathbb{C}$ is $J _{
\mathcal {A}} $ holomorphic. 
\item $J ^{ \mathcal {A}}
$ preserves the $\mathcal{A}$-horizontal distribution on $M \times \mathbb{C}$. 
\end{itemize}  
(We don't specify $ \{j _{r, \theta}\} $ in the notation for $J ^{ \mathcal
   {A}} $, the dependence will be implicit). For a family $ \{ \mathcal {A}
_{b}\} $ $\delta $-admissible with respect to $f $, $b \in B ^{k} $, we
define
\begin{equation} \label{eq.jbr} j _{b,r,\theta} = \psi _{b, r, \theta} ^{*} j,
\end{equation}
\begin{equation} \label{eq.psib} \psi_b: M \times \mathbb{C}
   \to M \times \mathbb{C}, 
\end{equation} $$\psi _{b} (x, r, \theta) =
(\gamma _{\theta} \circ f _{b, \theta} ^{-1} x, r, \theta)   \mbox { for } r
\geq 2, $$ $$\psi _{b} (x, r, \theta) = (x, r, \theta) \mbox { for } r \leq 1,
$$ with $\psi _{b}$ for $1 <r <2 $ being an interpolation determined by the
contraction of the loop $\gamma _{ \theta} \circ f _{b, \theta} ^{-1} $,
$\theta \in S ^{1} $ of Hamiltonian diffeomorphism, which is obtained by
concatenating $f (p_b)$, where $p _{b}$ is a smooth geodesic path constant near
end points from $b $ to $0 $ in $B ^{k} $ (for the flat metric), with a fixed
smooth path $p_0$ also constant near end points, from $f (0)$ to $\gamma$ in
$\mathcal{P}_{\phi}$.

 Let $$ \overline{\mathcal {M}} ( \{J ^{ \mathcal {A}} _{b}\}) \equiv
\overline{\mathcal {M}} ( \{J ^{ \mathcal {A}} _{b}\}, o _{\max}, [\sigma
_{\max}]) $$ denote the compactified space of pairs $ (u,b) $, for $u$ a class $ [\sigma _{\max}]$, $J _{b}
$-holomorphic section of $M \times \mathbb{C} $, with $u|_{ \{r\} \times S ^{1}  }$ 
asymptotic as $r \mapsto \infty $ to the flat
section $o _{\max,b} \in CF ({p ( \mathcal {A} _b)})$. Where the generator, 
% \begin{ \label{eq.omaxb}
 $o
_{\max,b}$
% \end{equation}
 corresponds to $o _{\max} \in CF (\gamma) $ under the canonical identification
 of generators of $CF (p ( \mathcal {A} _b))$ with those of $CF (p ( \mathcal
 {A}_0))=CF (\gamma)$, via the action of the loop $\gamma _{\theta} \circ f _{b,
 \theta} ^{-1} $, $ \theta \in S ^{1} $ of Hamiltonian diffeomorphisms. This
 identification is in fact an isomorphism of chain complexes $$CF (p (
 \mathcal {A} _{b}),  \{j _{b,r,\theta} \} _{r, \theta}) \to CF (\gamma, j), $$
 see Section \ref{section.prelim.floer}. 
The class $
[\sigma _{\max}] $ is the class of the section $z \mapsto x _{\max} $, for 
$x _{\max} $ the maximizer of $H _{t} ^{\gamma} $ as before.
An element $ (u,b)$ is said to be in class $[\sigma _{\max}]
$ if  the ``section'' $ \psi _{b} \circ u $ asymptotic to $o
_{\max} $, is in the class $[\sigma _{\max}] $, where ``class'' is now
unambiguous.

 For this class the virtual
dimension of $ \overline{\mathcal {M}} ( \{J_b ^{ \mathcal {A}}\}) $ will be 0. 
The compactification is the classical compactification in Floer theory.
 Elements in the boundary of $\overline{\mathcal {M}} ( \{J _{b} ^{ \mathcal
 {A}}\})$, may have vertical bubbles lying in the fibers $M $ of the projection
 $M \times \mathbb{C} \to \mathbb{C} $, and or may have  breaks as in usual
 Floer theory. These breaks happen in the $r>2$, flat, dilation invariant part
 of $M \times \mathbb{C}$. Projecting the ``section'' to $M $ in the $r>2 $
 region we get the usual picture for breaking of Floer trajectories. We will not
 say much more on this as this part of classical Floer theory.
 
  Since $M $ is monotone and  since the expected dimension of the moduli
space is $0 $, we may regularize so that $ \overline{\mathcal {M}} ( \{J
^{ \mathcal {A}}_b\}) $ consists only of smooth
curves. However, we will have to  deal with breaking (but not bubbling)  when
studying deformations of the data $ \{ \mathcal {A} _{b}\}$.

The map $\Psi _{\gamma} $ is defined as the Gromov-Witten invariant 
 $$\int_{ \overline{ \mathcal {M}} ( \{J   
^{reg, \mathcal {A}}_{b}\})} 1.$$ The fact that
regularization is possible via perturbation of the family $ \{ \mathcal {A}
_{b}\} $ is not immediate but readily follows by \cite[Theorem
8.3.1]{citeMcDuffSalamon$J$--holomorphiccurvesandsymplectictopology}. This
is going to be of paramount importance for the main argument. 
\begin{remark}
We need monotonicity as opposed to semi-positivity, as we have to deal with
families of almost complex structures on $M \times \mathbb{C}$.  The analogous
condition of being semi-positive that would be necessary is that for a generic
(in parametric sense) $k$-family of almost complex structures on $M \times
\mathbb{C} $ there are no vertical holomorphic spheres in $M \times \mathbb{C}$
with negative Chern number. Clearly this condition becomes more restrictive as
$k$ increases, on the other hand monotonicity insures this for all $k$ at once.  
\end{remark}
\begin{remark} The monotonicity assumption is not due to avoidance
of the virtual moduli cycle, it appears to be rather necessary for the argument
to go through at all. 
\end{remark}
 \begin{lemma} \label{lemma.defined} $\Psi _{\gamma} ( [f])$ is
independent of the choice of the family $ \{ \mathcal {A} _{b}\}$ admissible
with respect to $f' $, and of $f' \in [f] $.
\end{lemma}
To prove this we first need to show that for a deformation $ \{ \mathcal {A}
_{b,t}\} $, $ 0 \leq t \leq 1 $, there are no elements $$ (u,b,t)  \in 
\overline{\mathcal {M}} ( \{J ^{ \mathcal {A}} _{b,t}\}), $$
for $b $ near $ \partial B ^{k} $. For this we need the special nature of the
class $ [\sigma _{\max}] $. 
\begin{proposition} \label{lemma.induced} For $f $ as before, there is an
assignment $f \mapsto \mathcal {A} _{b,f}$, where $ \{ \mathcal {A} _{b,f}\} $ is admissible with respect to $f $. 
Denote by $ \{J ^{ \mathcal {A}} _{b,f}\} $ the induced family
of almost complex structures, then for a local unstable manifold $f _{\gamma} $ at $\gamma
$, the space $ \overline{ \mathcal {M}} ( \{J ^{ \mathcal {A}} _{b,f
_{\gamma}}\}) $ consists of only one point: $(\sigma _{\max},0 )$. 
\end{proposition}
\begin{theorem} \label{thm.robustustilovsky} Let $\gamma $ be an
Ustilovsky geodesic, s.t. the associated real linear Cauchy-Riemann operator
($rest$ stands for restricted: the full operator has $T_0B $ as a summand for
domain):
\begin{equation*} D  ^{rest}_{\sigma_{\max}}: \Omega ^{0} ( \sigma _{\max} ^{*}
T ^{vert} (M \times \mathbb{C}))  \to \Omega ^{0,1} (\sigma
_{\max} ^{*} T ^{vert} (M \times \mathbb{C})), 
\end{equation*} 
has no kernel. Let $f_{\gamma}:  (B^k, \partial B^k) \to
(\mathcal {P} _{\phi}, \mathcal {P} _{\phi,  E ^{\gamma}})$ be a local unstable
manifold at $\gamma $, then there is an admissible family $ \{ \mathcal {A}
_{b}\}$, $\mathcal{A} _{0} = \mathcal{A} _{0, f _{\gamma} }   $, such that 
$\overline{ \mathcal {M}} (\{J
^{\mathcal {A}} _{b}\}) $ consists only of $(\sigma _{\max},0 )$ and this
element is regular.
% In particular $$ \pm 1 = [\overline{ \mathcal {M}} ( \{J
% ^{reg, \mathcal {A}} _{b,f _{\gamma}}\}) ] \in H_0(\overline{ \mathcal
% {M}} ( \{J ^{ \mathcal {A}} _{b,f _{\gamma}}\}),
% \mathbb{Z}) = \mathbb{Z},$$
 And so if $\gamma$ is in addition robust then
 \begin{equation*} 0 \neq [f _{\gamma}] \in  \pi _{k} ( \mathcal {P} _{\phi},
 \mathcal {P} _{\phi,  E ^{\gamma}}).
 \end{equation*} 
\end{theorem}
% This implies that for any local unstable manifold $f _{\gamma} $ at $ \gamma $,
% (with $\gamma $ as above) we have that $$1 = [\overline{ \mathcal {M}} ( \{J
% ^{reg, \mathcal {A}} _{b,f _{\gamma}}\}) ] \in H_0(\overline{ \mathcal
% {M}} ( \{J ^{ \mathcal {A}} _{b,f _{\gamma}}\}),
% \mathbb{Z}) = \mathbb{Z}.$$
%  In particular
% $\Psi ( [f _{\gamma}]) =1$ 
The first half of the statement is the ``automatic transversality'',  
although the term is used in a somewhat looser sense than usual. The point
is that the full real linear CR operator with domain 
$\Omega ^{0} ( \sigma _{\max} ^{*} T
^{vert} (M \times \mathbb{C})) \oplus T_0 B $ may still have kernel on the $T_0
B $ component (and hence cokernel as the index is 0), but any such kernel is
``removable'' in the sense that there is a regularizing Fredholm perturbation of the Cauchy-Riemann section,
which does not change the 0 locus.  
The above theorem is the main ingredient for Theorem \ref{thm.virtual}.
\begin{proposition} \label{example.satisfied} For $\gamma $ as in
\ref{thm.special.case}, by the proof of \ref{thm.virtual}  the condition on the CR operator is always satisfied
for an almost complex structure $j$ on $M$ integrable and invariant under the action of $\gamma$ in a neighborhood of $x _{\max}$, and admitting a Kahler chart to $\mathbb{C} ^{n}$ at $x _{\max}$.  
\end {proposition}
\begin{proof}

 Pulling back the  $\gamma$ action to $\mathbb{R} ^{2n}$ by the Kahler chart at
 $x _{\max}$,
we get an action of $S ^{1}$ on a neighborhood $U$ of $0$ in $\mathbb{C}^{n}$ which preserves the standard complex structure and symplectic form, hence is an action by complex isometries of $U$ fixing $0 \in U$. Since such an isometry is linear, this determines a homomorphism $S ^{1} \to U (n)$.  The proof of \ref{thm.virtual} in this case gives that the normal bundle of $\sigma _{\max}$ is naturally holomorphic and a neighbhorhood of the 0-section is biholomorphic to a neighborhood of $\sigma _{\max}$ in $M \times \mathbb{C}$ with respect to the almost complex structure $J _{\mathcal{A}_{\gamma}}$. Completing the proof \ref {thm.virtual} we get the desired claim. 

\end{proof}
 \begin{theorem} \label{thm.minimizing} Under the conditions of theorem
\ref{thm.robustustilovsky} above, the local unstable manifold $$f _{\gamma}: (B^k, \partial B^k) \to (\mathcal {P} _{\phi},
\mathcal {P} _{\phi, E ^{\gamma}})$$ realizes the minimum in the definition of 
the semi-norm:
\begin{equation} | \cdot|: \pi _{k} ( \mathcal {P} _{\phi},
\mathcal {P} _{\phi, {E ^{\gamma}}}) \to \mathbb{R} ^{+},
\end{equation} 
given by $$| [f]| = \inf _{f' \in [f]} \max _{b \in B^k} L ^{+} (f' (b)). $$
\end{theorem}
The proof of Theorem \ref{thm.virtual} proceeds by
constructing $\gamma' $ from $\gamma $ satisfying the condition on the CR
operator.
 \section {Acknowledgements} I am particularly grateful to Yael Karshon who
 motivated me to pursue this topic over a number of years.  Egor
 Shelukhin for related discussions and Leonid Polterovich for helping at the
 conception state and who explained to me in particular some general aspects of
 calculus of variations. Michael Usher for helpful input as well as the
 referee for careful reading. 
 \section{Some preliminaries and notation}
 \subsection{Floer chain complex} \label{section.prelim.floer}
Classically, the generators of the Hamiltonian Floer chain complex associated to
 \begin{equation*} H: M \times S ^{1}  \to
\mathbb{R},
\end{equation*}
 are pairs $ (o, \overline{o}) $, with $o $ a time 1 periodic orbit of the
Hamiltonian flow generated by $H $, and $ \overline{o} $ a homotopy class of a
disk bounding the orbit. The function $H $ determines a Hamiltonian connection $ \mathcal {A} _{H}
$ on the bundle $M \times S ^{1} \to S ^{1}$. The horizontal spaces for $
\mathcal {A} _{H} $ are the $ \widetilde{\Omega} _{H} $
orthogonal spaces to the vertical tangent spaces of $M \times S ^{1} $, where 
\begin{equation*} \widetilde{\Omega} _{ H} = \omega - d (Hd \theta).  
\end{equation*}
To remind the reader our convention for the Hamiltonian flow is:
\begin{equation*}  \omega (X _{t}, \cdot)=-dH _{t} (
\cdot).
\end{equation*}
The horizontal (a.k.a flat) sections for $
\mathcal {A} _{H} $ correspond to periodic orbits of $H $, in the obvious way. The homotopy classes $ \overline{o}
$, induce homotopy classes of bounding disks in $M \times D ^{2} $ of the
corresponding flat sections. The connection $ \mathcal {A} _{H} $, induces an obvious $ \mathbb{R}
$-translation invariant connection on $M \times \mathbb{R} \times S ^{1} $, trivial in the $
\mathbb{R} $ direction. For an $ \mathbb{R}$ translation invariant family $
\{j _{\theta}\}$ of almost complex structures on the vertical tangent bundle of
$M \times \mathbb{R} $, we have an induced almost complex structure $J
_{H} $ on $ M \times \mathbb{R} \times S ^{1}$ as explained in the introduction. 
The differential in the classical Hamiltonian Floer chain complex is obtained via count of $J _{H}$-holomorphic sections $u$ of $M
\times \mathbb{R} \times S ^{1} \to \mathbb{R} \times S ^{1}$, whose
projections to $M \times S ^{1}$ are asymptotic in backward,  forward time to
generators $ (o_-, \overline{o} _{-}) $, respectively $(o_+, \overline{o} _{+}) 
$ of $CF(H)$, such that $ \overline{o}_- + [u] - \overline{o}_+ =0 \in \pi_2
(M \times S ^{2})$, with the obvious interpretation of this equation, and such
that $$CZ (o_+, \overline{o} _{+} ) - CZ (o_-, \overline{o} _{-}) = 1.$$ We
shall omit further details. The corresponding chain complex will be denoted 
by $CF (p, \{j _{\theta}\}) $, if $p$ is the path from $id$ to $\phi \in \text
{Ham}(M, \omega) $ generated by $H$. If $ \{j _{\theta}\}$ is independent
of $\theta $ we just write $CF (p, j) $, or just $CF (p)$ if $ \{j
_{\theta}\}$ is implicit.
\subsection{Coupling forms} \label{sec.prelim.coup}
This material appears in \cite[Chapter
6]{citeMcDuffSalamonIntroductiontosymplectictopology}, in slightly
more generality of locally Hamiltonian fibrations, so we review it here only
briefly. A Hamiltonian fibration is a smooth fiber bundle $$M \hookrightarrow P \to X,$$
with structure group $ \text {Ham}(M, \omega) $.  A \emph{coupling form}  for a Hamiltonian
fibration $M \hookrightarrow P \xrightarrow{p} X$, is a closed 2-form $
\widetilde{\Omega} $ whose restriction to fibers coincides with $\omega $ (with respect to a Hamiltonian trivialization), and
which has the property: 
\begin{equation*}  \int _{M}  \widetilde{\Omega} ^{n+1} =0 \in \Omega ^{2} (X).
\end{equation*}
Such a 2-form determines a Hamiltonian connection $ \mathcal {A} _{
\widetilde{\Omega}} $, by declaring horizontal spaces to be $ \widetilde{\Omega}
$ orthogonal spaces to the vertical tangent spaces. A Hamiltonian connection $
\mathcal {A} $ in turn determines a coupling form $ \widetilde{\Omega} _{
\mathcal {A}} $ as follows. First we ask that $ \widetilde{\Omega} _{ \mathcal
{A}} $ generates the connection $ \mathcal {A} $ as above. This determines $
\widetilde{\Omega} _{ \mathcal {A}} $, up to values on $ \mathcal {A}
$ horizontal lifts $ \widetilde{v},  \widetilde{w} \in T _{p} P $ of $v,w \in T
_{x} X $. We specify these values by the formula $$ \widetilde{\Omega} _{
\mathcal {A}} ( \widetilde{v}, \widetilde{w}) = R _{ \mathcal {A}} (v, w) (p),$$ 
where $R _{ \mathcal {A}}| _{x} $ is the curvature 2-form with values in $C
^{\infty} _{norm} (p ^{-1} (x))$. 
For a connection $\mathcal{A}_{H}$ induced by $H: M \times S ^{1} \to \mathbb{R}$, the associated coupling form $\widetilde{\Omega}_{\mathcal{A} _{H}}$ is given by:
\begin{equation} \label{eq.couplingequality}
	\widetilde{\Omega}_{\mathcal{A} _{H}} =  \widetilde{\Omega} _{H} = \omega - d (Hd \theta).
\end{equation}
In particular for a section $u$ of $M \times \mathbb{C}$ asymptotic to a flat section $o$, the integral of $\widetilde{\Omega}_{\mathcal{A} _{H}}$ over $u$
is the action of $o$ as a periodic orbit of $H$, (with bounding disk determined by $u$). 
\section{The proofs}

\begin{proof} [Proof of Lemma \ref{lemma.defined}] 
Let $ \mathcal {O} \to B ^{k} $, be a fibration with fiber over $b $ the space
of Hamiltonian connections $ \mathcal {A} $ on $M \times \mathbb{C} $, $\delta
$-admissible with respect to $f (b)$. By Proposition \ref{lemma.induced} the
fiber is non-empty, and is contractible (it is a $\delta$-ball in an affine
space), moreover it readily follows by Proposition \ref{lemma.induced} that $
\mathcal {O} $ is a Serre fibration. Consequently the space of sections of $ \mathcal {O} $ is connected by classical
obstruction theory. But this is exactly the space of families $ \{ \mathcal
{A} _{b}\} $, $\delta $-admissible with respect to $f$. 

Suppose we are given a one parameter family $ \{ \mathcal
{A} _{b,t}\} $, $0 \leq t \leq 1$, with each $ \mathcal {A} _{b,t} $ $\delta
$-admissible with respect to $f _{t} (b) $, $[f_t] = [f] $. To show that the
invariant $\Psi _{\gamma} ( [f]) $ is well defined, we need to show that a regular one parameter family induces
a one-dimensional compact cobordism 
$$ \overline{ \mathcal {M}} ( \{J
^{reg, \mathcal {A}} _{b, t, f_t }\}),$$ between 
$ \overline{ \mathcal {M}} ( \{J
^{reg, \mathcal {A}} _{b, 0, f_0 }\}),$ $\overline{
\mathcal {M}} ( \{J ^{reg, \mathcal {A}} _{b, 1, f_1 }\})$. We should elaborate on what regular means.  First the
chain complexes $CF (p ( \mathcal {A} _{b,t}, \{j _{b, t, r,\theta} \} _{r,
\theta} ) $, are meant to be regular for each $b,t$ but this follows by
construction of $\{j _{b, t, r,\theta} \} $ (analogous to construction of $ \{j
_{b,r, \theta}\},$ see  \eqref{eq.jbr}) assuming $j$ was taken to be
regular. Next, denote by ${\mathcal {B}}$ the space of triples $(u, b, t)$, $u
\in
\mathcal {B} _{b}$, $t \in [0,1] $ with $\mathcal {B}
_{b}$ denoting the space of class $[\sigma _{\max}]$-smooth sections of $M
\times \mathbb{C}$, asymptotic to $o _{\max, b} $. This is a Frechet
bundle over $B ^{k}$, the charts can be constructed using the diffeomorphisms $\{\psi _{b}
\} $, (see \eqref{eq.psib}). After appropriate Sobolev completions which we
don't specify (as this is classical), we get a Banach bundle
\begin{equation*} {\mathcal {E}} \to {\mathcal {B}},
\end{equation*}
whose fiber over $u _{b}=(u,b,t)$ is $\Omega ^{0,1} (S ^{2}, u ^{*} T ^{vert}(M
\times \mathbb{C} ))$, and the section we call $ \mathcal {F} _{f}$, 
\begin{equation*} \mathcal {F} _{f} (u _{b})= \delbar _{J  ^{ \mathcal
{A}}_{b,t}} (u).
\end{equation*}
The space ${ \mathcal {M}} ( \{J
^{\mathcal {A}} _{b, t, f_t }\}) $ is identified with  the 0-locus of this
section. As $ \mathcal {B}$ fibers over $B \times [0,1] $, the so called vertical differential at $u \in \mathcal {M} (J
^{\mathcal {A}} _{b, t, f_t }) $, is a real linear Cauchy-Riemann operator
of the form $$D_{u}: \Omega ^{0} ( u   ^{*} T ^{vert} (M \times
\mathbb{C})) \oplus T(B \times [0,1])  \to \Omega ^{0,1} (u   ^{*} T
^{vert} (M \times \mathbb{C})),$$
and is Fredholm of index 
\begin{equation*} CZ (o _{\max}) -CZ ( [M]) + \dim B +1 = 1, 
\end{equation*}
by assumption that $\dim B = \ind \gamma =  -(CZ (o _{\max}) -CZ ( [M])) $.
We say $u$ is regular if this operator is surjective. For a general $u \in
\overline {\mathcal {M}} (J ^{\mathcal {A}} _{b, t, f_t }),$ we will say that it
is regular if the analogue of the operator above is surjective for $u _{princ}
$:  the principal component of the ``section'' $u$, i.e. the component of
the holomorphic building not entirely contained in the translation invariant part of
$(M \times \mathbb{C}, \{ \mathcal {A} _{b,t}\}) $, and which is not a vertical
bubble. We say $ \{J
^{\mathcal {A}} _{b, t, f_t } \}$  is regular if all the elements of the
corresponding compactified moduli space are regular. 
This latter regularity can be obtained, by pertubing the family of connections $
\{ \mathcal {A} _{b,t}\}$,  \cite[Section
8]{citeMcDuffSalamon$J$--holomorphiccurvesandsymplectictopology}
 because the monotonicity assumption rules out holomorphic bubbles in 
 the fiber with negative Chern number.  Under above regularity vertical bubbling
 cannot happen. This is because a vertical bubble in class $A $ drops the Fredholm index of the CR
 operator at the principal component by $2 \langle c_1 TM, A \rangle \geq 2$,
 (monotonicity assumption)
 which would make the Fredholm index at any such principal component negative.
% $$D_{u _{princ}}: \Omega ^{0} ( u _{princ}  ^{*} T ^{vert} (M \times
% \mathbb{C})) \oplus T(B \times [0,1])  \to \Omega ^{0,1} (u _{princ}  ^{*} T
% ^{vert} (M \times \mathbb{C})),$$ to be surjective for each $ u \in  \overline{
% \mathcal {M}} ( \{J ^{reg, \mathcal {A}} _{b, t, f_t }\}) $, where $u _{princ}
% $ denotes In this case there is no vertical bubbling, but there maybe breaking. 
To show that we have such a cobordism (between $ \overline{ \mathcal {M}} ( \{J
^{reg, \mathcal {A}} _{b, 0, f_0 }\}),$ $\overline{
\mathcal {M}} ( \{J ^{reg, \mathcal {A}} _{b, 1, f_1 }\})$) we need two things. 
First that there is an $ \epsilon>0 $, s.t. there are no $ (u,b,t)  \in  
\overline{ \mathcal {M}} ( \{J ^{reg, \mathcal {A}} _{b, t, f_t }\}) $ with $b$ in the $\epsilon $-neighborhood
 of $ \partial B ^{k} $, denoted by $ \partial B ^{k} _{\epsilon}$. Take
 $\epsilon $ so that $f (\partial B ^{k} _{\epsilon}) \subset \mathcal {P}
 _{\phi, E _{\gamma} + \delta} - \mathcal {P} _{\phi, E _{\gamma}}$. By remark following \eqref{eq.couplingequality}, we have $[\widetilde{\Omega} _{ \mathcal
{A}, b,t}] ([u])= - L ^{+} (\gamma)$. 
 Then for $(u,b,t) \in \overline{ \mathcal {M}} ( \{J ^{reg, \mathcal {A}} _{b, t, f_t }\}),$ $b \in \partial B ^{k} _{\epsilon} $ we have:
\begin{equation} \label{eq.calculation} 0 \leq \langle[\widetilde{\Omega} _{ \mathcal
{A}, b,t} + \pi ^{*} (\alpha _{ \mathcal {A}, b,t})], [u] \rangle = - L ^{+} (\gamma) + \area ( \mathcal
{A} _{b,t}) < - L ^{+} (\gamma)  + E _{\gamma} + \delta + \delta <0, 
\end{equation}
for $\alpha _{ \mathcal {A}, b,t}$, as in \eqref{eq.alphaA}, but this is
impossible. 

 Next we need to show that the signed count of  boundary points
 of the manifold $$ \overline{ \mathcal {M}} ( \{J
^{reg, \mathcal {A}} _{b, t, f_t }\}) ,$$ corresponding to broken holomorphic
sections is 0. As $ \{J
^{reg, \mathcal {A}} _{b, t, f_t } \}$ is by assumption regular, a broken
holomorphic section $ (u,b,t) $ will have a pair of components (levels of the
holomorphic building): a principal component  $ (u_1, b,t)$ with $u _{1} $ a $J
^{reg, \mathcal {A}} _{b, t, f_t} $ holomorphic section of $M \times \mathbb{C}
$, asymptotic to a generator $g$ of $CF (p (\mathcal {A} _{b,t})) $ with
Conley-Zehnder index $CZ (o _{\max}) +1$, and a component corresponding to a
Floer gradient trajectory from $g $ to $o _{\max} $.   Regularity rules out
all other non-compactness possibilities by dimension counting. 

 But in this case either the signed count of flow lines
from $g $ to $o _{\max} $ is zero, which would be what we want or 
 $o _{\max} $ is not semi homologically
essential, which contradicts our hypothesis. 
\end {proof}

\begin{proof} [Proof of Proposition \ref{lemma.induced}] For a given  
$f:  (B^k, \partial B^k) \to (\mathcal {P} _{\phi}, \mathcal {P}
_{\phi,{E ^{\gamma}}}) $ and $b \in B ^{k} $, the construction of $ \mathcal
{A} _{b,f}$ is as follows. We first obtain a coupling form $\widetilde{\Omega} _{b,f} $
 on the trivial fibration $$M \times \mathbb{C} \to
\mathbb{C}.$$ This is a form with support in $\{r \geq 1\}
\subset \mathbb{C} $, such that under a fixed identification $$\{r \geq 1\}
\subset \mathbb{C}  \simeq \mathbb{R} ^{\geq 1} \times S ^{1},$$ 
it has the form
\begin{equation} \label{eq.Omegaf} \widetilde{\Omega} _{b,f}= \omega - d (\eta
H _{b} d \theta), 
\end{equation}
where $0 \leq r < \infty,  0 \leq \theta \leq 1 $, (recall that we are using modified polar coordinates) $H _{b} $ is the normalized generating function
for $ f (b) $, and
$\eta: \mathbb{R} \to [0,1]$ is a smooth, monotonely increasing function, with support in $ [1, \infty]
$ satisfying 
% $$0 \leq \eta' (r),$$ and 
\begin{equation} \label{eq.eta} \eta (r) = \begin{cases} 1 & \text{if } 2
-\kappa \leq r < \infty,\\ r-1  & \text{if } 1+ \kappa \leq r \leq 2-2\kappa,
\end{cases}
\end{equation}
  for a small $\kappa >0$. The Hamiltonian connection $ \mathcal
{A} _{b,f} $ is then defined by taking the horizontal spaces for $ \mathcal {A}
_{b,f} $ to be the $ \widetilde{\Omega} _{b} $ orthogonal spaces to the vertical
tangent spaces. 

If $f _{\gamma} $ is a local unstable manifold at $\gamma $, then for  $$ (u,b)
\in \overline{ \mathcal {M}} ( \{ J ^{\mathcal {A}} _{b, f _{\gamma} } \}), $$
for $b \neq 0$, $L ^{+}(f (b)) < L ^{+} (\gamma) $ and so 
\begin{equation*} 0 \leq \langle [ \widetilde{\Omega} _{b,f} + \pi ^{*} \alpha
_{ \mathcal {A} _{b,f}}],  [u] \rangle = - L ^{+} (\gamma) + L ^{+} (f (b)),
\end{equation*}
with the calculation $\pi ^{*} \alpha _{
\mathcal {A} _{b,f}} ( [u]) = L ^{+} (f (b))$, being elementary from
definitions, but this is impossible and so $b=0 $. We need to check that the
section $\sigma _{\max} $, is the only element of $ \overline{ \mathcal {M}} ( J
^{\mathcal {A}} _{0,f _{\gamma}}) $. It is simple to check that it is the only
smooth element, for given another smooth $u \in  \overline{ \mathcal {M}} ( J
^{\mathcal {A}} _{0,f _{\gamma}})  $ we have
\begin{multline} \label{eq.vanishing} 0=  \langle [\omega - \eta (r) dH_0
\wedge d \theta - H_0 d \eta \wedge d \theta + (max _{x \in M} H _0 (x)) d \eta
\wedge d \theta], [u] \rangle.
% \int _{u}
% \omega - \eta (r) d H \wedge d \theta  - \int _{u} H  d  \eta  \wedge d\theta +
% (\sup_x {H} (x))d \eta \wedge d\theta.
\end{multline}
Note that $u$ is necessarily horizontal, for otherwise
right hand side is positive by \eqref{eq.almostcomp}. Hence the
form $\omega - \eta (r) d H_0 \wedge d \theta$ must vanish on $u$, as
the horizontal subspaces are spanned by vectors $ \frac{\partial}{\partial r}, \frac{\partial}{\partial
\theta} + \eta (r)X _{H_0}$. Thus,
\eqref{eq.vanishing} can only happen if $u$ is $\sigma _{{\max}}$. 
Upon a moment of reflection we see that the same argument works for a general $u
\in \overline{ \mathcal {M}} ( J
^{\mathcal {A}} _{0,f _{\gamma}})  $. 
% \footnote{This argument is not very
% original, it appeared in slightly different form in the work of Seidel and
% Polterovich}

\end{proof} 
\begin{proof} [Proof of Theorem \ref{thm.robustustilovsky}]

By the assumption that $f _{\gamma} $ is Morse
at $\gamma$, and by \ref{lemma.induced}, $$ \overline{ \mathcal {M}} (
\{J ^{ \mathcal {A}} _{f _{\gamma}  ,b}\})$$ is a zero dimensional manifold 
consisting of a single point $\sigma _{\max}$. This is the expected dimension as 
the Fredholm index of 
$$D _{\sigma_{\max}}: \Omega ^{0} ( \sigma _{\max} ^{*} T
^{vert} (M \times \mathbb{C})) \oplus T_0B \to \Omega ^{0,1} (\sigma
_{\max} ^{*} T ^{vert} (M \times \mathbb{C})),$$
is 
\begin{equation*}  CZ (o _{\max})
- CZ ( [M]) + k = CZ (o _{\max}) - CZ ( [M])+\ind \gamma = 0,
\end{equation*}
 by Theorem \ref{thm.indexUst}. The restricted operator
$$D ^{rest}_{\sigma_{\max}}: \Omega ^{0} ( \sigma _{\max} ^{*} T
^{vert} (M \times \mathbb{C}))  \to \Omega ^{0,1} (\sigma
_{\max} ^{*} T ^{vert} (M \times \mathbb{C})),$$ has no kernel by assumption, 
and so the dimension of its cokernel is $$-(CZ (o _{\max})
- CZ ( [M])) = k.$$  
The point of the following construction is to perturb the family $\{\mathcal{A}
_{b,f _{\gamma} } \}$  so that this cokernel is covered by the $T_0B $ component of the
total vertical differential.

% and let $v_1, \ldots, v_l $ be a basis for $ \mathcal
% {V} $, $l \leq k $. 
Let $ \mathcal {H} $ denote the space of coupling forms of the form $
\widetilde{\Omega} _{\mathcal {A} _{p}} + \Pi $, where $ \mathcal {A} _{p} $ is the Hamiltonian connection
on $M \times \mathbb{C} $, induced by $p \in \mathcal {P} _{\phi} $ as in Proposition \ref{lemma.induced},
and $\Pi $ is of the form:
\begin{equation*} \widetilde{\Omega} _{p}  + d (G 
  d\theta),
\end{equation*}
$G  $  is a normalized function with  support in $ \{1 + \kappa
\leq r \leq 2-2 \kappa \} \subset \mathbb{C} $, ($\kappa $ as in the proof of
Proposition \ref{lemma.induced}). Such that identifying $ \{r \geq 1\} \subset \mathbb{C} $
with $ \mathbb{R} ^{\geq 1} \times S ^{1}$, $G$  has the form: 
\begin{equation} \label{eq.gb} G: M
\times \mathbb{R} ^{\geq 1} \times S ^{1} \to \mathbb{R}, \quad G _{r} =
 G | _{(M \times \{r\}) \times S ^{1}} = \zeta (r) \cdot K,
\end{equation}
where $K: M \times S ^{1} \to \mathbb{R} $, such that $K (x
_{\max})=0$  and $ \zeta: \mathbb{R} ^{\geq 0}
\to \mathbb{R}$ is a function with compact support. (Here $x _{\max} $ is the
extremizer of the generating function of $\gamma $ as before.) To emphasize we are not fixing $p$,$K$, $\zeta$.   
Let $\mathcal{C} _{\mathcal{H}}$ denote the associated space of Hamiltonian connections. 

%  $ \{ \mathcal { \mathcal {P}} _{ v _{i}}\} \in
% \mathcal {H}$ of $
% \{
% \mathcal {A} _{b,f}\} $, so that $$ D ^{univ}
% _{\sigma _{\max}, \mathcal {A}_0} (0 \oplus \mathcal {P} _{v _{i}}) =v_i.$$ And
% set $ \mathcal {P} _{v_i} =0$ for $l < i \leq k $. In fact by the argument of
% \cite[Theorem 8.3.1]{MS} we may arrange that the connection form for each $  \mathcal {P} _{v _{i}} $ 
% 
% We take $ \{ \mathcal {A} _{b}\} $, $ \mathcal {A} _{0} = \mathcal {A} _{0,f} $
% s.t. for the basis $ \{ \frac{d}{dx_i}\} $ of $T _{0}B $, $0 \leq i \leq k $ we
% have: $$ \frac{d}{dx_i} \{
% \mathcal {A} _{b}\} = \epsilon_i \cdot \mathcal {P} _{i},$$ $0 < \epsilon_i 
% $. 
\begin{lemma} \label{lemma.crit} Let $ \{ \mathcal {A} _{b}\}$, $ \mathcal {A}
_{b} \in \mathcal{C}_{\mathcal {H}} $, $ \mathcal {A} _{0} = \mathcal {A} _{\gamma}$ be a family of Hamiltonian
connections on $M \times \mathbb{C} $. Then we have $$ d \area _{ \{ \mathcal
{A} _{b}\}} (0) = 0,$$ where
$$\area _{ \{ \mathcal {A} _{b}\}} (b) = \area ( \mathcal {A} _{b}).$$ 
\end{lemma}
\begin{proof}  Suppose we have a variation $ \mathcal {A} _{b ( \tau)}$, with $ b (\tau) $, $ -1 \leq
\tau \leq 1 $ a smooth path through $0 \in B ^{k} $, $b (0)=0 $. Denote by $ x _{\max, \tau, r, \theta} $ the maximizer of $\ast R _{ \mathcal {A} _{b (\tau) }} (r, \theta) $, with the latter denoting the Hodge star (evaluate curvature on an
orthonormal pair) of the curvature 2-form of $ \mathcal {A} _{b (\tau)} $, at
$r, \theta$ (identifying the lie algebra with $C ^{\infty} _{norm} (M)$). Here the Hodge star is taken with respect to a metric
$g _{ \mathbb{C}} $ on $
\mathbb{C} $, for which the identification $\{r \geq 1\}
\subset \mathbb{C}  \simeq \mathbb{R} ^{r \geq
1} \times S ^{1}$ is an isometry, for the classical metric $g _{st}$ on the
latter. (I.e. $
\mathcal {A} _{b (\tau)} = \ast R _{ \mathcal {A} _{b (\tau) }} \otimes _{ \mathbb{R}} \omega _{st}$ for $\omega
_{st} $ the classical volume form on $ \mathbb{R} \times S ^{1} $, thinking of $\ast R _{ \mathcal {A} _{b (\tau) }}$ as a lie algebra valued function.)

 As $
\ast R _{ \mathcal {A} _{b (0)}} (r,\theta) =  H ^{\gamma} _{
\theta}$, for $\{1 + \kappa
\leq r \leq 2-2 \kappa \} $ is Morse at $x _{\max} $ the point $ x _{\max, \tau,
r, \theta} $ is uniquely determined and varies smoothly with $\tau $ for $\tau $ small. 

The derivative at $ \tau=0 $ of $\area (\tau) = \area ( \mathcal {A}
_{b (\tau)})$, is 
\begin {equation}
\begin{split} \area' (0) = \frac{d}{d \tau} \big| _{\tau =0}  L ^{+} (b (\tau))
+ \int _{ \mathbb{C}} dH ^{\gamma} _{x, \theta} (
\frac{d}{d \tau} \big | _{\tau =0} x _{\max, \tau, r, \theta}) \, dvol _{g _{
\mathbb{C}}} + \\ \int _{ \mathbb{C}} \frac{d}{d \tau} \big | _{\tau=0} K _{
\mathcal {A} _{b (\tau)}, r, \theta} (x _{\max})  \,dvol _{g _{ \mathbb{C}}}.
\end{split}
\end {equation}
For clarity we mention that the above expansion of the derivative comes from the chain rule for the composition 
\begin{equation*}
	\mathbb{R} \to \mathbb{R}^{3} \to C ^{\infty} (\mathbb{C}, \mathbb{R})
        \xrightarrow{\int _{\mathbb{C}}} 
           \mathbb{R},
\end{equation*}
with the obvious maps not indicated. 
 The first term vanishes as $b (0) = \gamma $ is
critical for $L ^{+} $ by assumption. The second term vanishes as $x _{\max} $
is critical for $H ^{\gamma} _{\theta}$ for each $\theta $ by assumption. The
last term vanishes by assumption that $K (x _{\max})=0 $ in \eqref{eq.gb}.
\end{proof}
Each element  in $ \mathcal {H} $ (or $\mathcal{C}_{H}$) determines  an almost complex structure on $M \times
\mathbb{C} $ as before, so we have the universal differential 
$$D ^{univ} _{\sigma _{\max}, \mathcal {A}_\gamma}:\Omega ^{0} ( \sigma _{\max} ^{*} T
^{vert} (M \times \mathbb{C})) \oplus T _{ \mathcal {A} _{\gamma}}\mathcal {H} \to  \Omega
^{0,1} (\sigma _{\max} ^{*} T ^{vert} (M \times \mathbb{C})),$$
which by the proof of \cite[Theorem
8.3.1]{citeMcDuffSalamon$J$--holomorphiccurvesandsymplectictopology} and\cite[Remark
3.2.3]{citeMcDuffSalamon$J$--holomorphiccurvesandsymplectictopology} is surjective, (we are dropping 
all mentions of Sobolev completions as this is standard.)

Let $ \mathcal {V}$ denote a fixed complement to the image $D _{\sigma_{\max}}$,
$ 0 \leq \dim \mathcal {V} \leq k $.  By the above we may find a subspace $
\mathcal {S} \subset \mathcal {H} $, $0 \leq \dim \mathcal {S} \leq k $ so that
$$ D ^{univ} _{\sigma _{\max}, \mathcal {A}_\gamma}  (0 \oplus \mathcal {S}) =
\mathcal {V}.$$ We now take $ \{ \mathcal {A} _{b}\}$,  $ \mathcal {A} _{b} \in
\mathcal {H} $, $ \mathcal {A} _{0} = \mathcal {A} _{\gamma}$, such that this is
a family  $\delta $-admissible with respect to $f$, such that the natural
differential $$T_0 B \to T _{ \mathcal {A} _{\gamma}} \mathcal {H} $$ is onto $
\mathcal {S} $, and such that $ \{ \mathcal {A} _{b}\} $ is sufficiently $C
^{\infty} $ close to $ \{ \mathcal {A} _{b,f}\} $ so that the function $\area _{
\{ \mathcal {A} _{b}\}}: B ^{k} \to \mathbb{R}$ is still Morse at  the unique
maximum $b=0 $.  This last condition is possible by Lemma \ref{lemma.crit} and
the following elementary observation. 

\begin{lemma} Suppose $f: B ^{k} \to
\mathbb{R} $ is a function Morse at the unique maximizer $0 \in B ^{k}$, and
$f'$ a smooth function sufficiently $C ^{\infty} $ close to $f$, and such that
$0$ is a critical point of $f' $. Then $0$ is also a unique maximizer of $f' $,
and moreover $f' $ is Morse at $0$. 
\end{lemma} 
\begin{proof} This follows by
Morse Lemma the proof is omitted.
 \end{proof} By the proof of Proposition
\ref{lemma.induced}, $ (\sigma _{\max}, 0)$ is the only element of $\overline{
\mathcal {M}} (\{J  ^{ \mathcal {A}}_{b}\})  $, and so this moduli space is
regular and the Gromov-Witten invariant is $ \pm 1 $. \end {proof} \begin{proof}
[Proof of Theorem \ref{thm.minimizing}]  Clearly $|f _{\gamma}| = L ^{+}
(\gamma)$. On the other hand if there is an $f' \in [f _{\gamma}] $, with $|f'|
< L ^{+} (\gamma)$, then by the proof of Lemma \ref{lemma.defined} the moduli
space $ \overline{ \mathcal {M}} ( \{J ^{ \mathcal {A}} _{f',b}\}) $ is empty
but this contradicts \ref{thm.robustustilovsky}. \end{proof} \begin{proof}
[Proof of theorem \ref{thm.virtual}] Given our robust Ustilovsky geodesic
$\gamma $, and $H ^{\gamma} _{\theta} $ its generating Hamiltonian, let $H': M
\times S ^{1} \to \mathbb{R}$ be a function  whose pull-back to  $ \mathbb{C}
^{n} $ by a fixed Darboux chart of $x _{\max} \in M$, coincides with the Hessian
of $H ^{\gamma}$ at $x _{\max} $: $Hess (H ^{\gamma} _{\theta}) (x _{\max})$ as
a function $ \mathbb{C} ^{n} \to \mathbb{R} $, for each $\theta $. Taking this
Darboux neighborhood to be suitably small the resulting $H'  $ can be made
arbitrarily $C ^{\infty} $  close to $H ^{\gamma} _{\theta} $.  Consequently the
resulting path $\gamma'  $ is still a robust Ustilovsky geodesic, (it may have a
different end point).  We may also suppose without loss of generality that the
Darboux chart   $ \mathbb{C} ^{n} \to M $, is   holomorphic with respect to $j $
on $M $. Let $ \mathcal {A}' $ denote the   Hamiltonian connection on $M \times
\mathbb{C} $, induced by $H'$.  It follows that a normal neighborhood of
$\sigma _{\max} \subset (M \times \mathbb{C}, J ^{ \mathcal {A'}})$, is
biholomorphic to a neighborhood of a 0 section of a $ \mathbb{C} ^{n} $ vector
bundle $E $ over $ \mathbb{C} $, whose almost complex structure is induced by
the unitary connection $A'$, determined by $$Hess (H ^{\gamma}) (x _{\max}):
\mathbb{C} ^{n} \times S ^{1} \to \mathbb{R},$$ (exactly as in Proposition
\ref{lemma.induced}). Such an almost complex structure must be integrable
\cite[Section 5]{citeAtiyahBottTheYang-MillsequationsoverRiemannsurfaces.}. Consequently,  the operator \begin{equation*} D
^{rest}_{\sigma_{\max}, \mathcal {A}'}: \Omega ^{0}  ( \sigma _{\max} ^{*} T
^{vert} (M \times \mathbb{C}))  \to \Omega ^{0,1} (\sigma _{\max} ^{*} T ^{vert}
(M \times \mathbb{C})),  \end{equation*} is the Dolbeault operator for the
holomorphic structure on $E $ above. If $\xi \neq 0$ is in the kernel of the
operator above, then since a normal neighborhood of $\sigma _{\max} \subset (M
\times \mathbb{C}, J ^{ \mathcal {A'}})$ is biholomorphic to a neighborhood of
the 0 section of $E $, there would be an element of $ \overline{ \mathcal {M}}
(\{J  ^{ \mathcal {A}'}\})  $, corresponding to $ \epsilon \cdot \xi $,
$\epsilon >0$ small enough,  and such an element would different from $\sigma
_{\max} $, which is impossible. The claim then follows by Theorem
\ref{thm.robustustilovsky}. \end{proof} \begin{proof} [Proof of Theorem
\ref{thm.index0}]  By Theorem \ref{thm.virtual} for any robust Ustilovsky
geodesic $\gamma $ there is an arbitrarily $C ^{\infty} $ close Ustilovsky
geodesic $\gamma'$, (with possibly different right end point) which is
minimizing in its homotopy class relative end points. It immediately follows
that $ \gamma $ itself must be minimizing in its homotopy class  relative end
points.  balkh\end{proof}

  \bibliographystyle{siam}
  \bibliography{/Users/yashasavelyev/GoogleDrive/workspace/link} 
\end{document}